\documentclass[english]{article}
\usepackage[T1]{fontenc}
\usepackage[latin9]{inputenc}
\usepackage{amsthm}
\usepackage{amssymb}

\makeatletter

\providecommand{\tabularnewline}{\\}

\newcommand{\lyxaddress}[1]{
\par {\raggedright #1
\vspace{1.4em}
\noindent\par}
}
\theoremstyle{plain}
\newtheorem{thm}{\protect\theoremname}
\ifx\proof\undefined
\newenvironment{proof}[1][\protect\proofname]{\par
\normalfont\topsep6\p@\@plus6\p@\relax
\trivlist
\itemindent\parindent
\item[\hskip\labelsep
\scshape
#1]\ignorespaces
}{%
\endtrivlist\@endpefalse
}
\providecommand{\proofname}{Proof}
\fi
\theoremstyle{plain}
\newtheorem{cor}[thm]{\protect\corollaryname}

\makeatother

\usepackage{babel}
\providecommand{\corollaryname}{Corollary}
\providecommand{\theoremname}{Theorem}

\begin{document}

\title{\noindent Congruent conditions on the number of terms, on the ratio
number of terms to first terms and on the difference of first terms
for sums of consecutive squared integers equal to squared integers}

\author{\noindent Vladimir Pletser}

\maketitle

\lyxaddress{\noindent European Space Research and Technology Centre, ESA-ESTEC
P.O. Box 299, NL-2200 AG Noordwijk, The Netherlands E-mail: Vladimir.Pletser@esa.int }
\begin{abstract}
\noindent {\normalsize{Sums of $M$ consecutive squared integers $\left(a+i\right)^{2}$
equaling squared integers (for $a\geq1$, $0\leq i\leq M-1$) yield
certain linear groupings of pairs $\left(a_{1},a_{2}\right)$ of $a$
values for successive same values of $M$ when these are linked by
$a_{1}+a_{2}=\mu M+1$ with $\mu=\left(\eta/\delta\right)\in\mathbb{\mathbb{Q}}^{+}$.
In this paper, congruent conditions on $M,\eta,\delta$, and on the
difference $\left(a_{2}-a_{1}\right)$ are demonstrated for these
linear groupings to hold. It is found that $\eta\equiv1\left(mod\,2\right)$}}
{\normalsize{and $\delta\equiv0,1$ or $5\left(mod\,6\right)$, and
if $\delta\equiv0\left(mod\,6\right)$, $M\equiv0\left(mod\,12\right)$,
while if $\delta\equiv1$ or $5\left(mod\,6\right)$, $M\equiv2$
or $11\left(mod\,12\right)$ with $a_{1}$ and $a_{2}$ being of different
or same parities.}}{\normalsize \par}

\noindent \textbf{Keywords}: Sums of consecutive squared integers
equal to square integers ; Congruence
\end{abstract}
MSC2010 : 11E25 ; 11A07

\section{\noindent Introduction}

\noindent Finding all values of $a\geq1$ for which the sum of $M>1$
consecutive integer squares starting from $a^{2}$ equals an integer
square $s^{2}$ was addressed by several authors since Lucas proposed
the initial problem for $a=1$ in 1873 (see e.g. \cite{key-1-24,key-1-15,key-1-17,key-11,key-1-23,key-1-26,key-1-4}).
More recently, the present author showed \cite{key-10-1,key-10-2}
that there are no integer solutions if $M\equiv3,5,6,7,8$ or $10\left(mod\,12\right)$;
that there are integer solutions if $M$ is not a square integer and
is congruent to $0,9,24$ or $33\left(mod\,72\right)$, or to $1,2$
or $16\left(mod\,24\right)$, or to $11\left(mod\,12\right)$; and
if $M$ is a square integer and congruent to $1\left(mod\,24\right)$.

\noindent Interestingly, for certain values of $a$ and $M$, one
finds linear groupings of pairs $\left(a_{1},a_{2}\right)$ of $a$
values for successive same values of $M$ when these values are linked
by the relation $a_{1}+a_{2}=\mu M+1$ with $\mu=\left(\eta/\delta\right)\in\mathbb{\mathbb{Q}}^{+}$
\cite{key-5-2}. 

\noindent In this paper, congruent conditions on $M,\eta,\delta$,
and on the difference $\left(a_{2}-a_{1}\right)$ are demonstrated
for these groupings to hold.

\section{\noindent Pairs of $a$ values}

\noindent For $1\leq j\leq2,i,\eta,\delta,M_{\mu,k}>1,a_{j,\mu,k}\in\mathbb{Z}^{*}$,
$k,s_{j,\mu,k}\in\mathbb{Z}$, let $\mu=\left(\eta/\delta\right)\in\mathbb{\mathbb{Q}}^{+}$
forming an irreducible fraction. Two values $a_{1,\mu,k},a_{2,\mu,k}$
form a pair of $a_{j,\mu,k}$ values for a same value of $M_{\mu,k}$
if 
\begin{eqnarray}
a_{1,\mu,k}+a_{2,\mu,k} & = & \mu M_{\mu,k}+1\\
a_{2,\mu,k}-a_{1,\mu,k} & = & f_{\mu,k}
\end{eqnarray}
It was demonstrated \cite{key-5-2} that if 
\begin{equation}
M_{\mu,k}=\frac{\delta^{2}\left(3f_{\mu,k}^{2}-1\right)}{3\left(\eta+\delta\right)^{2}+\delta^{2}}\label{eq:20}
\end{equation}

\noindent then the sums of $M_{\mu,k}$ consecutive squared integers
$\left(a_{j,\mu,k}+i\right)^{2}$ always equal squared integers $s_{j,\mu,k}^{2}$
for pairs of $a_{j,\mu,k}$ values

\noindent 
\begin{equation}
\sum_{i=0}^{M-1}\left(a_{j,\mu,k}+i\right)^{2}=M_{\mu,k}\left[\left(a+\frac{M_{\mu,k}-1}{2}\right)^{2}+\frac{M_{\mu,k}^{2}-1}{12}\right]\label{eq:5}
\end{equation}

\noindent $\forall k\in\mathbb{Z}$ and where $j=1$ or $2$.

\section{\noindent Congruent values of $M_{\mu,k}$, $f_{\mu,k}$ and $\mu=\left(\eta/\delta\right)$}

\noindent These results hold only for certain allowed values of $M_{\mu,k}$,
of $f_{\mu,k}$ and of $\mu=\left(\eta/\delta\right)\in\mathbb{\mathbb{Q}}^{+}$,
that can be determined as follows. Relation (\ref{eq:20}) reads also
\begin{equation}
\left(\delta f_{\mu,k}\right)^{2}-M_{\mu,k}\left(\eta+\delta\right)^{2}=\delta^{2}\left(\frac{M_{\mu,k}+1}{3}\right)\label{eq:41-1-1}
\end{equation}
As the sum of $M$ consecutive integer squares can only be equal to
a squared integer if $M\equiv0\left(mod\,12\right)$ (more precisely
$M\equiv0$ or $24\left(mod\,72\right)$) or if $M\equiv1,2$ or $4\left(mod\,12\right)$
(more precisely $M\equiv1,2$ or $16\left(mod\,24\right)$), or if
$M\equiv9\left(mod\,12\right)$ (more precisely $M\equiv9$ or $33\left(mod\,72\right)$),
or if $M\equiv11\left(mod\,12\right)$ \cite{key-10-1,key-10-2},
the following theorem constrains the congruent values of $\eta,\delta,M_{\mu,k}$
and $f_{\mu,k}$.
\begin{thm}
\noindent For $\eta,\delta,M_{\mu,k}>1,f_{\mu,k}\in\mathbb{Z}^{+}$,
$k\in\mathbb{Z}$, for (\ref{eq:41-1-1}) to hold:

\noindent \textup{$\eta\equiv1\left(mod\,2\right)$} and $\delta\equiv0,1$
or $5\left(mod\,6\right)$, and

\noindent - if $\delta\equiv0\left(mod\,6\right)$, $M_{\mu,k}\equiv0\left(mod\,12\right)$, 

\noindent - if $\delta\equiv1$ or $5\left(mod\,6\right)$, $M_{\mu,k}\equiv2$
or $11\left(mod\,12\right)$ for $f_{\mu,k}\equiv1$ or $0\left(mod\,2\right)$; 

\noindent more precisely, $\eta,\delta,M_{\mu,k}$ and $f_{\mu,k}$
are congruent to the values of Table 1.
\begin{table}
\begin{centering}
\caption{Congruent values of $\eta,\delta,M_{\mu,k}$ and $f_{\mu,k}$ for
$\delta\equiv0\left(mod\,6\right)$ and $\delta\equiv1$ or $5\left(mod\,6\right)$ }

\par\end{centering}

\begin{centering}
\begin{tabular}{|c|c||c||c|}
\hline 
$\delta\equiv$ & $\eta\equiv$ & $f_{\mu,k}\equiv$ & $M_{\mu,k}\equiv$\tabularnewline
\hline 
\hline 
$0\left(mod\,36\right)$ & $1$ or $5\left(mod\,6\right)$ & $\forall$ & $0\left(mod\,144\right)$ \tabularnewline
\cline{1-1} \cline{4-4} 
$12$ or $24\left(mod\,36\right)$ &  &  & $96\left(mod\,144\right)$\tabularnewline
\hline 
$6$ or $30\left(mod\,36\right)$ & $1$ or $5\left(mod\,6\right)$ & $1\left(mod\,2\right)$ & $24\left(mod\,144\right)$\tabularnewline
\cline{1-1} \cline{4-4} 
$18\left(mod\,36\right)$ &  &  & $72\left(mod\,144\right)$\tabularnewline
\hline 
\hline 
$1\left(mod\,6\right)$ & $1$ or $3\left(mod\,6\right)$ & $1$ or $5\left(mod\,6\right)$ & $50\left(mod\,72\right)$\tabularnewline
\cline{2-2} \cline{4-4} 
 & $5\left(mod\,6\right)$ &  & $2\left(mod\,72\right)$\tabularnewline
\cline{2-4} 
 & $1$ or $3\left(mod\,6\right)$ & $3\left(mod\,6\right)$ & $2\left(mod\,72\right)$\tabularnewline
\cline{2-2} \cline{4-4} 
 & $5\left(mod\,6\right)$ &  & $26\left(mod\,72\right)$\tabularnewline
\hline 
$5\left(mod\,6\right)$ & $1\left(mod\,6\right)$ & $1$ or $5\left(mod\,6\right)$ & $2\left(mod\,72\right)$\tabularnewline
\cline{2-2} \cline{4-4} 
 & $3$ or $5\left(mod\,6\right)$ &  & $50\left(mod\,72\right)$\tabularnewline
\cline{2-4} 
 & $1\left(mod\,6\right)$ & $3\left(mod\,6\right)$ & $26\left(mod\,72\right)$\tabularnewline
\cline{2-2} \cline{4-4} 
 & $3$ or $5\left(mod\,6\right)$ &  & $2\left(mod\,72\right)$\tabularnewline
\hline 
$1\left(mod\,6\right)$ & $1$ or $3\left(mod\,6\right)$ & $0\left(mod\,6\right)$ & $11\left(mod\,36\right)$\tabularnewline
\cline{2-2} \cline{4-4} 
 & $5\left(mod\,6\right)$ &  & $35\left(mod\,36\right)$\tabularnewline
\cline{2-4} 
 & $1$ or $3\left(mod\,6\right)$ & $2$ or $4\left(mod\,6\right)$ & $23\left(mod\,36\right)$\tabularnewline
\cline{2-2} \cline{4-4} 
 & $5\left(mod\,6\right)$ &  & $11\left(mod\,36\right)$\tabularnewline
\hline 
$5\left(mod\,6\right)$ & $1\left(mod\,6\right)$ & $0\left(mod\,6\right)$ & $35\left(mod\,36\right)$\tabularnewline
\cline{2-2} \cline{4-4} 
 & $3$ or $5\left(mod\,6\right)$ &  & $11\left(mod\,36\right)$\tabularnewline
\cline{2-4} 
 & $1\left(mod\,6\right)$ & $2$ or $4\left(mod\,6\right)$ & $11\left(mod\,36\right)$\tabularnewline
\cline{2-2} \cline{4-4} 
 & $3$ or $5\left(mod\,6\right)$ &  & $23\left(mod\,36\right)$\tabularnewline
\hline 
\end{tabular}
\par\end{centering}

\begin{centering}
i.e. for $\delta\equiv0\left(mod\,36\right)$ and $\eta\equiv1$ or
$5\left(mod\,6\right)$, $\forall f_{\mu,k}$ and $M_{\mu,k}\equiv0\left(mod\,144\right)$; 
\par\end{centering}

\centering{}for $\delta\equiv1\left(mod\,6\right)$ and $\eta\equiv1$
or $3\left(mod\,6\right)$, $f_{\mu,k}\equiv1$ or $5\left(mod\,6\right)$
and $M_{\mu,k}\equiv50\left(mod\,72\right)$ 
\end{table}
\end{thm}
\begin{proof}
\noindent For $\eta,\delta,d,M_{\mu,k}>1,f_{\mu,k},m,n\in\mathbb{Z}^{+}$,
$\lambda\in\mathbb{Z}^{*}$, $k\in\mathbb{Z}$, recalling that $n^{2}\equiv0,1,4$
or $9\left(mod\,12\right)$, $\forall n\in\mathbb{Z}^{+}$, and that
$M_{\mu,k}\equiv0,1,2,4,9$ or $11\left(mod\,12\right)$ \cite{key-10-1},

\noindent (i) Let $\delta\equiv0\left(mod\,6\right)$ $\Rightarrow$
$\delta^{2}\equiv0\left(mod\,12\right)$, then only $\left(\eta+\delta\right)^{2}\equiv1\left(mod\,12\right)$
holds for, if $\left(\eta+\delta\right)^{2}\equiv0,4$ or $9\left(mod\,12\right)$,
$\Rightarrow$ $\left(\eta+\delta\right)\left(mod\,6\right)\equiv\eta\left(mod\,6\right)\equiv0,\left(2,4\right)$
or $3\left(mod\,6\right)$ as $\delta\equiv0\left(mod\,6\right)$
and $\left(\eta/\delta\right)$ would not be an irreducible fraction;
therefore $\left(\eta+\delta\right)^{2}\equiv1\left(mod\,12\right)$
and $\eta\equiv1$ or $5\left(mod\,6\right)$.

\noindent Then (\ref{eq:41-1-1}) yields the congruence relation $M_{\mu,k}\equiv0\left(mod\,12\right)$,
and more precisely $M_{\mu,k}\equiv0$ or $24\left(mod\,72\right)$
\cite{key-10-1}, i.e. $\exists m\in\mathbb{Z}^{+}$ such that $M_{\mu,k}=24\left(3m+\lambda\right)$,
with $\lambda=0$ or $1$ if $M_{\mu,k}\equiv0$ or $24\left(mod\,72\right)$.

\noindent As $\delta\equiv0\left(mod\,6\right)$, $\exists d\in\mathbb{Z}^{+}$
such that $\delta=6d$ and $\delta^{2}=36d^{2}$. Then (\ref{eq:41-1-1})
reads after dividing by $12$,
\begin{equation}
3d^{2}f_{\mu,k}^{2}-2\left(3m+\lambda\right)\left(\eta+6d\right)^{2}=d^{2}\left(24\left(3m+\lambda\right)+1\right)\label{eq:27-1}
\end{equation}
that yields the congruence relation
\begin{equation}
3d^{2}f_{\mu,k}^{2}-2\left(3m+\lambda\right)-d^{2}\equiv0\left(mod\,12\right)\label{eq:35-3}
\end{equation}

\noindent (i.1) If $d^{2}\equiv0\left(mod\,12\right)$, i.e. $d\equiv0\left(mod\,6\right)$,
then (\ref{eq:35-3}) reduces to 

\noindent $10\left(3m+\lambda\right)\equiv0\left(mod\,12\right)$
$\Rightarrow$ $\left(3m+\lambda\right)\equiv0\left(mod\,6\right)$
which cannot hold if $\lambda=1$ and which yields $m\equiv0\left(mod\,2\right)$
if $\lambda=0$, i.e. $M_{\mu,k}\equiv0\left(mod\,144\right)$, and
$\forall f_{\mu,k}$ with $\delta\equiv0\left(mod\,36\right)$.

\noindent (i.2) If $d^{2}\equiv1\left(mod\,12\right)$, i.e. $d\equiv1$
or $5\left(mod\,6\right)$, then (\ref{eq:35-3}) reads
\begin{equation}
3f_{\mu,k}^{2}-2\left(3m+\lambda\right)\equiv1\left(mod\,12\right)\label{eq:30-1}
\end{equation}

\noindent (i.2.1) If $f_{\mu,k}^{2}\equiv0$ or $4\left(mod\,12\right)$,
then (\ref{eq:30-1}) reduces to $10\left(3m+\lambda\right)\equiv1\left(mod\,12\right)$
which cannot hold whether $\lambda=0$ or $1$; 

\noindent (i.2.2) if $f_{\mu,k}^{2}\equiv1$ or $9\left(mod\,12\right)$,
then (\ref{eq:30-1}) reduces to $2\left(3m+\lambda\right)\equiv2\left(mod\,12\right)$
$\Rightarrow$ $\left(3m+\lambda\right)\equiv1\left(mod\,6\right)$
which cannot hold if $\lambda=0$ and which yields $m\equiv0\left(mod\,2\right)$
if $\lambda=1$, i.e. $M_{\mu,k}\equiv24\left(mod\,144\right)$, and
$f_{\mu,k}\equiv1\left(mod\,2\right)$ with $\delta\equiv6$ or $30\left(mod\,36\right)$.

\noindent (i.3) If $d^{2}\equiv4\left(mod\,12\right)$, i.e. $d\equiv2$
or $4\left(mod\,6\right)$, then (\ref{eq:35-3}) reads 

\noindent $10\left(3m+\lambda\right)\equiv4\left(mod\,12\right)$
$\Rightarrow$ $\left(3m+\lambda\right)\equiv4\left(mod\,6\right)$
which cannot hold if $\lambda=0$ and which yields $m\equiv1\left(mod\,2\right)$
if $\lambda=1$, i.e. $M_{\mu,k}\equiv96\left(mod\,144\right)$, and
$\forall f_{\mu,k}$ with $\delta\equiv12$ or $24\left(mod\,36\right)$.

\noindent (i.4) If $d^{2}\equiv9\left(mod\,12\right)$, i.e. $d\equiv3\left(mod\,6\right)$,
then (\ref{eq:35-3}) reads
\begin{equation}
3f_{\mu,k}^{2}-2\left(3m+\lambda\right)\equiv9\left(mod\,12\right)\label{eq:31}
\end{equation}

\noindent (i.4.1) If $f_{\mu,k}^{2}\equiv0$ or $4\left(mod\,12\right)$,
then (\ref{eq:31}) reduces to $10\left(3m+\lambda\right)\equiv9\left(mod\,12\right)$
which cannot hold whether $\lambda=0$ or $1$; 

\noindent (i.4.2) if $f_{\mu,k}^{2}\equiv1$ or $9\left(mod\,12\right)$,
then (\ref{eq:31}) reduces to $2\left(3m+\lambda\right)\equiv6\left(mod\,12\right)$
$\Rightarrow$ $\left(3m+\lambda\right)\equiv3\left(mod\,6\right)$
which cannot hold if $\lambda=1$ and which yield $m\equiv1\left(mod\,2\right)$
if $\lambda=0$, i.e. $M_{\mu,k}\equiv72\left(mod\,144\right)$, and
$f_{\mu,k}\equiv1\left(mod\,2\right)$ with $\delta\equiv18\left(mod\,36\right)$.

\noindent (ii) Let now $\delta\equiv3\left(mod\,6\right)$ $\Rightarrow$
$\delta^{2}\equiv9\left(mod\,12\right)$, then only $\left(\eta+\delta\right)^{2}\equiv1$
or $4\left(mod\,12\right)$ hold for, if $\left(\eta+\delta\right)^{2}\equiv0$
or $9\left(mod\,12\right)$, $\Rightarrow$ $\left(\eta+\delta\right)\equiv0$
or $3\left(mod\,6\right)$ $\Rightarrow$ $\eta\equiv3$ or $0\left(mod\,6\right)$
as $\delta\equiv3\left(mod\,6\right)$ and $\left(\eta/\delta\right)$
would not be an irreducible fraction; therefore $\left(\eta+\delta\right)^{2}\equiv1$
or $4\left(mod\,12\right)$.

\noindent Relation (\ref{eq:41-1-1}) yields then the congruence relation
\begin{equation}
9f_{\mu,k}^{2}-M_{\mu,k}\left(\left(\eta+\delta\right)^{2}+3\right)\equiv3\left(mod\,12\right)\label{eq:35-3-1}
\end{equation}

\noindent (ii.1) If $\left(\eta+\delta\right)^{2}\equiv1\left(mod\,12\right)$,
(\ref{eq:35-3-1}) reduces to $9f_{\mu,k}^{2}-4M_{\mu,k}\equiv3\left(mod\,12\right)$
which cannot hold whether $f_{\mu,k}^{2}\equiv0,1,4$ or $9\left(mod\,12\right)$.

\noindent (ii.2) If $\left(\eta+\delta\right)^{2}\equiv4\left(mod\,12\right)$,
(\ref{eq:35-3-1}) yields $9f_{\mu,k}^{2}-7M_{\mu,k}\equiv3\left(mod\,12\right)$; 

\noindent (ii.2.1) if $f_{\mu,k}^{2}\equiv0$ or $4\left(mod\,12\right)$,
it reduces to $M_{\mu,k}\equiv3\left(mod\,12\right)$; 

\noindent (ii.2.2) if $f_{\mu,k}^{2}\equiv1$ or $9\left(mod\,12\right)$,
it reduces to $M_{\mu,k}\equiv6\left(mod\,12\right)$; 

\noindent both cases have to be rejected as $M_{\mu,k}$ cannot be
congruent to $3$ or $6\left(mod\,12\right)$ \cite{key-10-1}.

\noindent (iii) If $\delta\equiv1$ or $2\left(mod\,3\right)$, then
for (\ref{eq:41-1-1}) to hold in integer values of $\eta,\delta,M_{\mu,k}$
and $f_{\mu,k}$, $M_{\mu,k}\equiv2\left(mod\,3\right)$ for the right
hand term of (\ref{eq:41-1-1}) to be integer. As $M_{\mu,k}$ can
only be congruent to $0,1,2,4,9$ or $11\left(mod\,12\right)$, $M_{\mu,k}\equiv2\left(mod\,3\right)$
can only be $M_{\mu,k}\equiv2\left(mod\,12\right)$ (and more precisely
$M_{\mu,k}\equiv2\left(mod\,24\right)$) or $M_{\mu,k}\equiv11\left(mod\,12\right)$
\cite{key-10-1}. It yields then that :

\noindent if $M_{\mu,k}\equiv2\left(mod\,24\right)$, $\left(\left(M_{\mu,k}+1\right)/3\right)\equiv1,9$
or $5\left(mod\,12\right)$ for $M_{\mu,k}\equiv2,26$ or $50\left(mod\,72\right)$,
and 

\noindent if $M_{\mu,k}\equiv11\left(mod\,12\right)$, $\left(\left(M_{\mu,k}+1\right)/3\right)\equiv4,8$
or $0\left(mod\,12\right)$ for $M_{\mu,k}\equiv11,23$ or $35\left(mod\,36\right)$.

\noindent (iii.1) Let $\delta^{2}\equiv1\left(mod\,12\right)$, i.e.
$\delta\equiv1$ or $5\left(mod\,6\right)$, then (\ref{eq:41-1-1})
yields
\begin{equation}
M_{\mu,k}\left(\eta+\delta\right)^{2}+\left(\frac{M_{\mu,k}+1}{3}\right)-f_{\mu,k}^{2}\equiv0\left(mod\,12\right)\label{eq:28}
\end{equation}
and one obtains as follows one of the two solutions :

\noindent Solution 1: $\left(\eta+\delta\right)^{2}\equiv0\left(mod\,12\right)$
$\Rightarrow$ $\left(\eta+\delta\right)\equiv0\left(mod\,6\right)$ 

\noindent $\Rightarrow$ $\eta\equiv5\left(mod\,6\right)$ or $1$
as $\delta\equiv1$ or $5\left(mod\,6\right)$,

\noindent Solution 2: $\left(\eta+\delta\right)^{2}\equiv4\left(mod\,12\right)$
$\Rightarrow$ $\left(\eta+\delta\right)\equiv2$ or $4\left(mod\,6\right)$ 

\noindent $\Rightarrow$ $\eta\equiv1$ or $3\left(mod\,6\right)$
if $\delta\equiv1\left(mod\,6\right)$ and $\eta\equiv3$ or $5\left(mod\,6\right)$
if $\delta\equiv5\left(mod\,6\right)$.

\noindent (iii.1.1) Let $f_{\mu,k}^{2}\equiv0$ or $4\left(mod\,12\right)$,
i.e. $f_{\mu,k}\equiv0\left(mod\,2\right)$, then :

\noindent - if $M_{\mu,k}\equiv2\left(mod\,12\right)$ with $\left(\left(M_{\mu,k}+1\right)/3\right)\equiv1,5$
or $9\left(mod\,12\right)$, (\ref{eq:28}) cannot hold; 

\noindent - if $M_{\mu,k}\equiv11\left(mod\,12\right)$ and

\noindent --- if $\left(\left(M_{\mu,k}+1\right)/3\right)\equiv0\left(mod\,12\right)$
(for $M_{\mu,k}\equiv35\left(mod\,36\right)$), (\ref{eq:28}) yields
Solution 1 for $f_{\mu,k}^{2}\equiv0\left(mod\,12\right)$ and $\left(\eta+\delta\right)^{2}\equiv8\left(mod\,12\right)$,
which cannot hold, for $f_{\mu,k}^{2}\equiv4\left(mod\,12\right)$; 

\noindent --- if $\left(\left(M_{\mu,k}+1\right)/3\right)\equiv4\left(mod\,12\right)$
(for $M_{\mu,k}\equiv11\left(mod\,36\right)$), (\ref{eq:28}) yields
Solution 2 or 1 for $f_{\mu,k}^{2}\equiv0$ or $4\left(mod\,12\right)$; 

\noindent --- if $\left(\left(M_{\mu,k}+1\right)/3\right)\equiv8\left(mod\,12\right)$
(for $M_{\mu,k}\equiv23\left(mod\,36\right)$), (\ref{eq:28}) yields 

\noindent $\left(\eta+\delta\right)^{2}\equiv8\left(mod\,12\right)$,
which cannot hold, for $f_{\mu,k}^{2}\equiv0\left(mod\,12\right)$
and Solution 2 for $f_{\mu,k}^{2}\equiv4\left(mod\,12\right)$.

\noindent (iii.1.2) Let $f_{\mu,k}^{2}\equiv1$ or $9\left(mod\,12\right)$,
i.e $f_{\mu,k}\equiv1\left(mod\,2\right)$, then:

\noindent - if $M_{\mu,k}\equiv2\left(mod\,12\right)$ and

\noindent --- if $\left(\left(M_{\mu,k}+1\right)/3\right)\equiv1\left(mod\,12\right)$
(for $M_{\mu,k}\equiv2\left(mod\,72\right)$), (\ref{eq:28}) yields
Solution 1 or 2 for $f_{\mu,k}^{2}\equiv1$ or $9\left(mod\,12\right)$; 

\noindent --- if $\left(\left(M_{\mu,k}+1\right)/3\right)\equiv5\left(mod\,12\right)$
(for $M_{\mu,k}\equiv50\left(mod\,72\right)$), (\ref{eq:28}) yields
Solution 2 for $f_{\mu,k}^{2}\equiv1\left(mod\,12\right)$ and $\left(\eta+\delta\right)^{2}\equiv2$
or $8\left(mod\,12\right)$, which cannot hold, for $f_{\mu,k}^{2}\equiv9\left(mod\,12\right)$; 

\noindent --- if $\left(\left(M_{\mu,k}+1\right)/3\right)\equiv9\left(mod\,12\right)$
(for $M_{\mu,k}\equiv26\left(mod\,72\right)$), (\ref{eq:28}) yields 

\noindent $\left(\eta+\delta\right)^{2}\equiv2$ or $8\left(mod\,12\right)$,
which cannot hold, for $f_{\mu,k}^{2}\equiv1\left(mod\,12\right)$
and Solution 1 for $f_{\mu,k}^{2}\equiv9\left(mod\,12\right)$. 

\noindent - if $M_{\mu,k}\equiv11\left(mod\,12\right)$ with $\left(\left(M_{\mu,k}+1\right)/3\right)\equiv0,4$
or $8\left(mod\,12\right)$, then (\ref{eq:28}) cannot hold.

\noindent (iii.2) Let now $\delta^{2}\equiv4\left(mod\,12\right)$,
i.e. $\delta\equiv2$ or $4\left(mod\,6\right)$, then (\ref{eq:41-1-1})
yields
\begin{equation}
M_{\mu,k}\left(\eta+\delta\right)^{2}+4\left(\frac{M_{\mu,k}+1}{3}\right)-4f_{\mu,k}^{2}\equiv0\left(mod\,12\right)\label{eq:28-2}
\end{equation}
and one obtains as follows one of the two solutions, both to be rejected
as $\left(\eta/\delta\right)$ would not be an irreducible fraction
:

\noindent Solution 3: $\left(\eta+\delta\right)^{2}\equiv0\left(mod\,12\right)$
$\Rightarrow$ $\left(\eta+\delta\right)\equiv0\left(mod\,6\right)$ 

\noindent $\Rightarrow$ $\eta\equiv4$ or $2\left(mod\,6\right)$
as $\delta\equiv2$ or $4\left(mod\,6\right)$,

\noindent Solution 4: $\left(\eta+\delta\right)^{2}\equiv4\left(mod\,12\right)$
$\Rightarrow$ $\left(\eta+\delta\right)\equiv2$ or $4\left(mod\,6\right)$

\noindent $\Rightarrow$ $\eta\equiv0$ or $2\left(mod\,6\right)$
if $\delta\equiv2\left(mod\,6\right)$ and $\eta\equiv4$ or $2\left(mod\,6\right)$
if $\delta\equiv4\left(mod\,6\right)$.

\noindent (iii.2.1) Let $f_{\mu,k}^{2}\equiv0$ or $9\left(mod\,12\right)$,
then:

\noindent - if $M_{\mu,k}\equiv2\left(mod\,12\right)$ and 

\noindent --- if $\left(\left(M_{\mu,k}+1\right)/3\right)\equiv1\left(mod\,12\right)$,
(\ref{eq:28-2}) yields Solution 4;

\noindent --- if $\left(\left(M_{\mu,k}+1\right)/3\right)\equiv5$
or $9\left(mod\,12\right)$, (\ref{eq:28-2}) cannot hold;

\noindent - if $M_{\mu,k}\equiv11\left(mod\,12\right)$ and 

\noindent --- if $\left(\left(M_{\mu,k}+1\right)/3\right)\equiv0\left(mod\,12\right)$,
(\ref{eq:28-2}) yields Solution 3 

\noindent --- if $\left(\left(M_{\mu,k}+1\right)/3\right)\equiv4$
or $8\left(mod\,12\right)$, (\ref{eq:28-2}) cannot hold.

\noindent (iii.2.2) Let $f_{\mu,k}^{2}\equiv1$ or $4\left(mod\,12\right)$,
then:

\noindent - if $M_{\mu,k}\equiv2\left(mod\,12\right)$ and

\noindent --- if $\left(\left(M_{\mu,k}+1\right)/3\right)\equiv1\left(mod\,12\right)$,
(\ref{eq:28-2}) yields Solution 3;

\noindent --- if $\left(\left(M_{\mu,k}+1\right)/3\right)\equiv5$
or $9\left(mod\,12\right)$, (\ref{eq:28-2}) cannot hold.

\noindent - if $M_{\mu,k}\equiv11\left(mod\,12\right)$ and

\noindent --- if $\left(\left(M_{\mu,k}+1\right)/3\right)\equiv0\left(mod\,12\right)$,
(\ref{eq:28-2}) cannot hold;

\noindent --- if $\left(\left(M_{\mu,k}+1\right)/3\right)\equiv4$
or $8\left(mod\,12\right)$, (\ref{eq:28-2}) yields Solution 3 or
4.

\noindent Therefore, the congruence of Table 1 hold.
\end{proof}
\noindent The following congruent relations for $M_{\mu,k}$ also
hold.
\begin{cor}
\noindent For $\eta,\delta,M_{\mu,k}>1,f_{\mu,k}\in\mathbb{Z}^{+}$,
$k\in\mathbb{Z}$:

\noindent (i) if \textup{$\delta\equiv1$ or $5\left(mod\,6\right)$,
}then\textup{ $M_{\mu,k}\equiv0\left(mod\,\delta^{2}\right)$; }

\noindent \quad{}\enskip{}if\textup{ $\delta\equiv0\left(mod\,6\right)$,
}then\textup{ $M_{\mu,k}\equiv0\left(mod\,\left(\delta^{2}/3\right)\right)$;}

\noindent (ii) if $f_{\mu,k}\equiv1\left(mod\,2\right)$, then, if
\textup{$\delta\equiv1$ or $5\left(mod\,6\right)$, $M_{\mu,k}\equiv0\left(mod\,2\delta^{2}\right)$
}and,\textup{ }

\noindent \quad{}\enskip{}if\textup{ $\delta\equiv0\left(mod\,6\right)$,
}then\textup{ $M_{\mu,k}\equiv0\left(mod\,\left(2\delta^{2}/3\right)\right)$.}\end{cor}
\begin{proof}
\noindent For $\eta,\delta,M_{\mu,k}>1,f_{\mu,k}\in\mathbb{Z}^{+}$,
$k\in\mathbb{Z}$:

\noindent (i) from (\ref{eq:20}), as $\gcd\left(\eta,\delta\right)=1$
and as $M_{\mu,k}$ must be integer, if $\delta\equiv1$ or $5\left(mod\,6\right)$,
then $\left(3\left(\eta+\delta\right)^{2}+\delta^{2}\right)$ must
divide $\left(3f_{\mu,k}^{2}-1\right)$, therefore $M_{\mu,k}\equiv0\left(mod\,\delta^{2}\right)$;
if $\delta\equiv0\left(mod\,6\right)$, then $\left(\left(\eta+\delta\right)^{2}+\left(\delta^{2}/3\right)\right)$
must divide $\left(3f_{\mu,k}^{2}-1\right)$, therefore $M_{\mu,k}\equiv0\left(mod\,\left(\delta^{2}/3\right)\right)$;

\noindent (ii) immediate from (i) and Theorem 1.
\end{proof}

\section{Conclusions}

\noindent For pairs $\left(a_{1,\mu,k},a_{2,\mu,k}\right)$ of $a$
values, the sums of $M_{\mu,k}$ consecutive squared integers starting
with $a_{1,\mu,k}$ or $a_{2,\mu,k}$ are always equal to squared
integers $s_{1,\mu,k}^{2}$ or $s_{2,\mu,k}^{2}$ $\forall k\in\mathbb{Z}$
if $M_{\mu,k}=\delta^{2}\left(3\left(a_{2,\mu,k}-a_{1,\mu,k}\right)^{2}-1\right)/\left(3\left(\eta+\delta\right)^{2}+\delta^{2}\right)$
with $\mu=\left(\eta/\delta\right)\in\mathbb{\mathbb{Q}}^{+}$. It
was proved that $\eta\equiv1\left(mod\,2\right)$ and $\delta\equiv0,1$
or $5\left(mod\,6\right)$ and in addition :

\noindent - if $\delta\equiv0\left(mod\,6\right)$, then $M_{\mu,k}\equiv0\left(mod\,12\right)$
(more precisely, $M_{\mu,k}\equiv0,24,72$ or $96\left(mod\,144\right)$),
and if $\delta\equiv1$ or $5\left(mod\,6\right)$, depending on $a_{1,\mu,k}$
and $a_{2,\mu,k}$ having different or same parities, then $M_{\mu,k}\equiv2$
or $11\left(mod\,12\right)$ (more precisely, $M_{\mu,k}\equiv2,26$
or $50\left(mod\,72\right)$ or $M_{\mu,k}\equiv11,23$ or $35\left(mod\,36\right)$);

\noindent - if $\delta\equiv1$ or $5\left(mod\,6\right)$, then $M_{\mu,k}\equiv0\left(mod\,\delta^{2}\right)$
and, if in addition $a_{1,\mu,k}$ and $a_{2,\mu,k}$ have different
parities, $M_{\mu,k}\equiv0\left(mod\,2\delta^{2}\right)$ ; if $\delta\equiv0\left(mod\,6\right)$,
then $M_{\mu,k}\equiv0\left(mod\,\left(\delta^{2}/3\right)\right)$
and, if in addition $a_{1,\mu,k}$ and $a_{2,\mu,k}$ have different
parities, $M_{\mu,k}\equiv0\left(mod\,\left(2\delta^{2}/3\right)\right)$.


\begin{thebibliography}{10}
\bibitem[1]{key-11} U. Alfred, Consecutive integers whose sum of
squares is a perfect square, Mathematics Magazine, 19-32, 1964.

\bibitem[2]{key-1-4} L.Beeckmans, Squares Expressible as Sum of Consecutive
Squares, The American Mathematical Monthly, Vol. 101, No. 5, 437-442,
May 1994.

\bibitem[3]{key-1-26} M. Laub, Squares Expressible as a Sum of n
Consecutive Squares, Advanced Problem 6552, American Mathematical
Monthly, 97, 622-625, 1990.

\bibitem[4]{key-1-24} E. Lucas, Recherches sur l'Analyse Indeterminée,
Moulins, p. 90, 1873.

\bibitem[5]{key-1-15} E. Lucas, Question 1180, Nouvelles Annales
de Mathématiques, Série 2, 14, 336, 1875.

\bibitem[6]{key-1-17} E. Lucas, Solution de la Question 1180, Nouvelles
Annales de Mathématiques, Série 2, 15, 429-432, 1877.

\bibitem[7]{key-1-23} S. Philipp, Note on consecutive integers whose
sum of squares is a perfect square, Mathematics Magazine, 218-220,
1964.

\bibitem[8]{key-10-1} V. Pletser, Congruence conditions on the number
of terms in sums of consecutive squared integers equal to squared
integers, ArXiv, http://arxiv.org/abs/1409.7969, 29 September 2014.

\bibitem[9]{key-10-2} V. Pletser, Additional congruence conditions
on the number of terms in sums of consecutive squared integers equal
to squared integers, ArXiv, http://arxiv.org/pdf/1409.6261v1.pdf,
20 August 2014.

\bibitem[10]{key-5-2} V. Pletser, Linear features in the plot of
the number of terms of sums of consecutive squared integers equal
to squared integers versus first terms in the sums, submitted, 21
August 2014.\end{thebibliography}
\end{document}